\documentclass[11pt]{article}
\usepackage{amsmath,amsfonts,amssymb,amsthm}
\usepackage[affil-it]{authblk}
\usepackage{framed}
\usepackage[dvipsnames]{color}
\usepackage[T1]{fontenc}
\usepackage{times} 
\usepackage[cp1250]{inputenc}  
\advance\topmargin by -\headheight
\advance\topmargin by -\headsep
\evensidemargin=\oddsidemargin
\makeatother
\newtheorem{theo}{Theorem}[section]
\newtheorem{prop}[theo]{Proposition}
\newtheorem{lemm}[theo]{Lemma}
\newtheorem{coro}[theo]{Corollary}
\newtheorem{defi}[theo]{Definition}

\numberwithin{equation}{section}
\newcommand{\A}{\mathcal{A}} 
\newcommand{\T}{\mathbb{T}} 
\newcommand{\R}{\mathbb{R}} 
\newcommand{\Z}{\mathbb{Z}} 
\renewcommand{\H}{\mathcal{H}} 
\renewcommand{\a}{\alpha} 
\newcommand{\oh}{{\tfrac{1}{2}}} 
\newcommand{\sq}{\unskip\nobreak\kern5pt\nobreak\vrule height4pt width4pt depth0pt} 
\newbox\ncintdbox \newbox\ncinttbox
\setbox0=\hbox{$-$} \setbox2=\hbox{$\displaystyle\int$}
\setbox\ncintdbox=\hbox{\rlap{\hbox
    to \wd2{\kern-.1em\box2\relax\hfil}}\box0\kern.1em}
\setbox0=\hbox{$\vcenter{\hrule width 4pt}$}
\setbox2=\hbox{$\textstyle\int$}
\setbox\ncinttbox=\hbox{\rlap{\hbox
    to \wd2{\kern-.14em\box2\relax\hfil}}\box0\kern.1em}

\title{Wodzicki residue and minimal operators on a noncommutative 4-dimensional torus}
\author{Andrzej Sitarz${}^{1,2}$ \thanks{Partially supported by NCN grant 2011/01/B/ST1/06474.}}
\affil{${}^1$ Institute of Physics, Jagiellonian University,\\
Reymonta 4, 30-059 Krak\'ow, Poland,}
\affil{${}^2$ Institute of Mathematics of the Polish Academy of Sciences, \\ 
\'Sniadeckich 8, Warszawa, 00-950 Poland.} 
\begin{document}
\maketitle
\begin{abstract}
We compute the Wodzicki residue of the inverse of a conformally rescaled Laplace 
operator over a $4$-dimensional noncommutative torus. We show that the straightforward 
generalization of the Laplace-Beltrami operator to the noncommutative case is not the 
minimal operator.
\end{abstract}
\thispagestyle{empty}
\section{Introduction}

Noncommutative geometry as proposed in \cite{Co94} aims to use geometric methods to study 
noncommutative algebras in a similar way differential  geometry is used to study spaces. One of the appealing potential applications is its use in physics to describe the structure of space-time and fundamental 
interactions at high energies. Although the construction of basic data in noncommutative  geometry is 
equivalent in the classical case to the standard data of Riemannian geometry \cite{Co10} in the genuine noncommutative examples this aspect has been not explored sufficiently until recently. 

In a series of papers \cite{FaKh1}-\cite{FaWo} a conformally rescaled metric has been proposed and 
studied for the noncommutative two and four-tori. This led to the expressions of Gauss-Bonnet theorem and formulas for the noncommutative counterpart of scalar curvature. Independently, another class of Dirac operators and metrics with the geometric interpretation as arising from the $U(1)$ connections on noncommutative circle bundles has been proposed by the author and L.Dabrowski in \cite{DaSi}. Following 
this lead, a more general type of metric on the two torus has been proposed and studied perturbatively 
for the two-torus \cite{DaSi2}. 

However, in all of the mentioned approaches one major assumption was made. The second term 
of the heat-kernel asymptotics of the rescaled Laplace operator (or square of the Dirac operator) was 
identified as a linear functional of the scalar curvature. In Riemannian geometry this is certainly true, 
provided that the Laplace operator (or Dirac operator) comes from the Levi-Civita connection. In the 
presence of the nontrivial torsion the term is modified and includes the integral of the square of the 
torsion, as has been observed already in the early computations of Seeley-Gilkey-de Witt coefficients. 
As in the noncommutative geometry there is no implicit notion of torsion, one may wonder whether 
the rescaled Laplace operators are those, which minimize the second term of the heat-kernel 
expansion for a fixed metric. Additionally, in the classical case the computation of second heat 
kernel coefficient is closely related to the computations of the Wodzicki residue of a certain power 
of the Dirac operator.

As it has been shown \cite{FaWo} and more generally in \cite{LeJiPa} Wodzicki residue exists 
also in the case of the pseudodifferential calculus over noncommutative tori. 

In this note we shall address the question of the minimal operators (using the Wodzicki residue to
check minimality) and compute the Wodzicki residue for a class of operators on the noncommutative 
torus.
 
\section{Noncommutative tori and their pseudodifferential calculus}

We use the usual presentation of the algebra of $d$-dimensional noncommutative torus as 
generated by $d$ unitary elements $U_i$, $i=1,\ldots,d$, with the relations 
$$ U_j U_k = e^{2\pi i\theta_{jk}} U_k U_j,$$ 
where  $0< \theta_{jk} < 1$ is real. The smooth algebra $\A(\T^d_\theta)$ 
is then taken as an algebra of elements
$$ a = \sum_{\beta \in \Z^d} a_{\beta} U^\beta,, $$
where $a_{\beta}$ is a rapidly decreasing sequence and
$$ U^\beta = U_1^{\beta_1} \cdots U_d^{\beta_d}. $$

The natural action of $U(1)^d$ by automorphisms, gives, in its infinitesimal
form, two linearly independent derivations on the algebra: given on the generators as:
\begin{equation}
\delta_k (U_j) = \delta_{jk}  U_j,
\label{deriv}
\end{equation}
where $\delta_{jk}$ denotes the Kronecker delta.

The canonical trace on $\A(\T^d_\theta)$ is
$$ \mathfrak{t}(a) = \a_{{\mathbf 0}}, $$
where ${\mathbf 0} = \{0,0,\ldots,0\} \in \Z^d$. The trace is invariant  
with respect to the action of $U(1)^d$, hence
$$ \mathfrak{t}(\delta_j(a)) = 0, \quad \forall j =1,\ldots,d.$$

By $\H$ we denote the Hilbert space of the GNS construction with respect to the trace $\mathfrak{t}$ on 
the $C^\ast$ completion of $\A(\T^d_\theta)$ and $\pi$ the associated faithful representation. The elements
of the smooth algebra $\A(\T^d_\theta)$ act on $\H$ as bounded operators by left multiplication, whereas
the derivations $\delta_i$ extend to densely defined selfadjoint operators on $\H$ with the 
smooth elements of the Hilbert space, $\A(\T^d_\theta)$, in their common domain.
\subsection{Pseudodifferential operators on $\T^d_\theta$}

The symbol calculus defined in \cite{CoTr} and developed further in 
\cite{CoMo10} (see also \cite{LeJiPa}) is easily generalized to the $d$-dimensional
case and to the operators defined above. We shall briefly review the basic definitions 
and methods used further in the note. Let us recall that a differential operator of order 
at most $n$ is of the form
$$ P = \sum_{0 \leq k \leq n} \sum_{|\beta_k|=k} a_{\beta_k} \delta^{\beta_k}, $$
where $a_{jk}$ are assumed to be in the algebra $\A(\T^d_\theta)$, $\beta_k \in \Z^d$
and:
$$ |\beta_k| = \beta_1+ \cdots + \beta_d, \;\;\;\;\;
\delta^\beta =  \delta_1^{\beta_1} \cdots \delta_d^{\beta_d}.$$ 
Its symbol is:
$$ \rho({P}) =   \sum_{0 \leq k \leq n} \sum_{|\beta_k|=k} a_{\beta_k} \xi^{\beta_k}, $$
where
$$ \xi^{\beta} =  \xi_1^{\beta_1} \cdots \xi_d^{\beta_d}.$$ 
On the other hand, let $\rho$ be a symbol of order $n$, which is assumed
to be a $C^\infty$ function from $\R^d$ to $\A(\T^d_\theta)$, which is homogeneous 
of order $n$, satisfying certain bounds (see \cite{CoTr} for details). With every such
symbol $\rho$ there is associated an operator $ P_\rho$ on a dense subset of $\H$
spanned by elements $a \in \A(\T^d_\theta)$:
$$ P_\rho(a) = \frac{1}{(2\pi)^d} 
\int_{\R^d \times \R^d} e^{-i \sigma \cdot \xi} \rho(\xi) \alpha_\sigma(a)\, d\sigma d \xi,$$
where
$$ \alpha_\sigma(U^\alpha) = e^{i \sigma \cdot \alpha} U^\alpha, \;\;\; \sigma \in \R^d, \alpha \in \Z^d. $$

For two operators $P,Q$ with symbols:
$$ \rho(P) = \sum p_\alpha \xi^\alpha, \;\;\; \rho(Q) = \sum q_\beta \xi^\beta, $$
we use the formula, which follows directly from the same computations as in the
case of classical calculus of pseudodifferential operators:

\begin{equation}
\rho(P Q) = \sum_\gamma \frac{1}{\gamma !} \partial_\xi^\gamma(\rho(P)) 
\delta^\gamma(\rho(Q)), \label{syprod}
\end{equation}
where $\gamma ! = \gamma_1 ! \cdots \gamma_d !$. 
 

\subsection{Wodzicki Residue}

In this part we shall provide an elementary proof that there exists a trace on the above
defined algebra of symbols on the $d$-dimensional noncommuttaive torus:

\begin{prop} \label{wprop}
Let $\rho = \sum_{j \leq k} \rho_j(\xi)$ be a symbol over the noncommutative torus
$\A(\T^d_\theta)$. Then the functional:
$$ \rho \mapsto \int_{S^{d-1}} \mathfrak{t} \left( \rho_{-d} (\xi) \right) d\xi, $$
is a trace over the algebra of symbols.
 \end{prop}

Let us start with a simple lemma about homogeneous functions.

\begin{lemm} \label{hlem}
Let $f$ be a smooth function on $\R^d \setminus \{0\}$, homogeneous of degree $\rho$. 
Then 
$$  \int_{S^{d-1}} \partial_{\xi^i} f(\xi) d\xi =0, $$
holds for every $1 \leq i \leq d$ if and only if $\rho=1-d$.
\end{lemm}
\begin{proof}
The $\Rightarrow$ part is trivial, as it is sufficient to take $f_j(\xi) = \xi^j (\xi^2)^{\frac{\rho-1}{2}}$. Then:
$$ \partial_i ( \xi^j (\xi^2)^{\frac{\rho}{2}-1} ) = \delta^{ij} (\xi^2)^{\frac{\rho-1}{2}} 
+ 2 \frac{\rho-1}{2} \xi^j \xi^i (\xi^2)^{\frac{\rho-1}{2}-1}, $$
which, when restricted to the sphere $\xi^2=1$ and integrated, gives:
$$ V(d) \left( 1 + (\rho-1) \frac{1}{d} \right), $$
where $V(d)$ is the volume of $d\!-\!1$ dimensional sphere. This vanishes only if $\rho=1\!-\!d$.

Assume now that we have a homogeneous function on $\R^d$ of degree $\rho$, denoted $f$. 
Observe that using the freedom of the choice of coordinates we can safely assume that $i=1$. 
Using the spherical coordinates $r,\phi_1,\ldots,\phi_{d-1}$:
$$ \xi^1 = r \sin \phi_1, \xi^2 = r \cos \phi_1 \sin \phi_2, \ldots, \xi^d = r \cos\phi_1 \cdots \cos \phi_{d-2} \cos \phi_{d-1},$$
we know the volume form on the sphere:
$$ \omega = (\sin \phi_1)^{d-2} (\sin \phi_2)^{d-3} \cdots \sin \phi_{d-2} \, d\phi_1 \cdots d\phi_{d-1}, $$
and we can express the partial derivative $\frac{\partial}{\partial x^1}$ as:
$$ \frac{\partial}{\partial \xi^1} = \cos \phi_1 \frac{\partial}{\partial r} 
    -\frac{\sin \phi_1}{r} \frac{\partial}{\partial \phi_1}. $$
    
Since the function $f$ is homogeneous in $r$ of order $\alpha$ we have:

$$  \frac{\partial f}{\partial \xi^1}  =   \cos \phi_1 \, \frac{\alpha}{r}  f  
    - \frac{\sin \phi_1}{r}\, \frac{\partial f}{\partial \phi_1}. $$

Consider now the following function in the coordinates $\phi_1,\ldots \phi_{d-1}$:
$$ \omega \left( \frac{\partial f}{\partial \xi^1} \right)_{|_{r=1}} = 
     \left(  \alpha (\sin \phi_1)^{d-2} (\cos \phi_1)  f_{|_{r=1}} 
           - (\sin \phi_1)^{d-1} \frac{\partial f_{|_{r=1}}}{\partial \phi_1}
      \right)   (\sin \phi_2)^{d-3} \cdots \sin \phi_{d-2} = \cdots $$
If $\alpha = (1\!-\!d)$ it could be written as:

$$  \cdots = \frac{\partial}{\partial \phi_1} \left( - (\sin \phi_1)^{d-1} f_{|_{r=1}} \right) 
(\sin \phi_2)^{d-3} \cdots (\sin \phi_{d-2}).$$
Since the integral of a function $f$ over the sphere $S^{d-1}$ in the spherical coordinates is:
$$ \int_{S^{d-1}} F = \int_0^{2\pi} d\phi_1 \int_0^\pi d\phi_2 \cdots \int_0^\pi d\phi_{d-1} 
(\sin \phi_1)^{d-2} (\sin \phi_2)^{d-3} \cdots \sin \phi_{d-2} \,  f(\phi_1,\ldots,\phi_{d-1}),$$
we have that:

$$ 
\begin{aligned} 
\int_{S^{d-1}} & \left( \frac{\partial f}{\partial \xi^1} \right)_{|_{r=1}} = \\
=& \int_0^\pi \!\!  d\phi_{d-1} \int_0^\pi \!\!  d\phi_{d-2} \sin\phi_{d-2} \cdots \int_0^\pi \!\!   d\phi_2 (\sin \phi_2)^{d-3} 
 \int_0^{2\pi} \!\!  d\phi_1 \! \!  \left( \frac{\partial}{\partial \phi_1} \left( - (\sin \phi_1)^{d-1} f_{|_{r=1}} \right) \right) 
 = 0.
\end{aligned}
$$ 
\end{proof}

\begin{proof}[Proof of Proposition \ref{wprop}]
Let $\rho$ and $\sigma$ be two symbols, that is $C^\infty$ maps from $\R^d$ to $\A(\T^d_\theta)$, which
are decomposed into the sum of homogeneous symbols:
$$ \sigma = \sum_{j \leq S} \sigma_j, \;\;\;\; \rho = \sum_{j \leq R} \rho_j. $$
We shall prove that the the Wodzicki residue is a trace, that is:
$$ \hbox{Wres}(\rho \sigma) = \hbox{Wres}(\sigma \rho). $$

Using the formula for the product (\ref{syprod}) we have:
$$ (\rho \sigma)_{-d} = \sum_{\stackrel{\gamma,j,k}{|\gamma|+j+k=-d}} 
\frac{1}{\gamma !} \partial_\xi^\gamma(\rho_j)  \delta^\gamma(\sigma_k),$$
and since derivation in $\xi$ decreases the degree of homogeneity by $1$, the sum
is necessarily finite. 

First we shall prove the trace property for $|\gamma|=0$ and $|\gamma|=1$. If $|\gamma|=0$
we have:

$$  \mathfrak{t} \left( \rho_j \sigma_k  \right) =  \mathfrak{t} \left(\sigma_k   \rho_j \right), $$
since $\mathfrak{t}$ is a trace on the algebra $\A(\T^d_\theta)$. 

Next, if $|\gamma|=1$, we have:

$$ 
\begin{aligned}
\int_{S^{d-1}} d\xi\, \mathfrak{t}\left( \partial_{\xi^i}(\rho_j) \delta_i(\sigma_k)  \right) 
&= \int_{S^{d-1}} d\xi \, \mathfrak{t}\left( - \delta_i(\partial_{\xi^i}(\rho_j)) \sigma_k  \right) \\
&= \int_{S^{d-1}} d\xi \, \mathfrak{t}\left( - \partial_{\xi^i}(\delta_i(\rho_j)) \sigma_k  \right) \\
&= \int_{S^{d-1}} d\xi \, \mathfrak{t}\left( \delta_i(\rho_j)  \partial_{\xi^i}(\sigma_k)  \right) \\
&= \int_{S^{d-1}} d\xi \, \mathfrak{t}\left(  \partial_{\xi^i}(\sigma_k)  \delta_i(\rho_j) \right),
\end{aligned}
$$
where we have used that $\mathfrak{t}$ is invariant with respect to $U(1)$ symmetry generated by
$\delta_i$,  the fact that $\delta_i$ and $\partial_{\xi^i}$ commute, then Lemma \ref{hlem} (which we can 
use because the product $\sigma_j \rho_j$ is homogeneous of degree $1\!-\!d$), and finally the trace 
property of $\mathfrak{t}$.

For any $|\gamma|>1$ we repeat the above argument sufficient number of times. 
\end{proof}

\section{Laplace-type operator of a conformally rescaled metric}

In this section we shall fix our attention on a family of Laplace-type operators, which originate 
from a conformally rescaled fixed metric on manifold. We begin with the classical situation. 
Let us take a closed Riemannian manifold $M$ of dimension $d$ with a fixed metric tensor $g$. 
If $h$ is a positive function on $M$ then we take the conformally rescaled metric to be given by:
$$ g_{ab} \to h^2 g_{ab} = \tilde{g}_{ab}, $$
where $g_{ab}$ is the original metric tensor.
\begin{lemm}\label{L1}
Let $\Delta$ be the usual Laplace operator on $M$ with the metric given by the metric tensor 
$g_{ab}$ and $\H$ be the Hilbert space of $L^2(M,g)$ (where the measure is taken with 
respect to the metric $g_{ab}$).  Let $\tilde{\Delta}$ be the Laplace operator 
on $M$ with the conformally rescaled metric acting on the Hilbert space 
$\tilde{\H} = L^2(M, \tilde{g})$. 

Then $\tilde{\Delta}$ is unitarily equivalent to $\Delta_h = 
h^{\frac{d}{2}} \tilde{\Delta} h^{-\frac{d}{2}}$ acting on
$\H$. Moreover, the operator $\Delta_h$ written in local coordinates is:
$$ 
\Delta_h = h^{-2} \Delta - 2 h^{-3} g^{ab} (\partial_a h) \partial_b 
    + h^{\frac{d}{2}-2} (\Delta h^{-\frac{d}{2}}). 
 $$
\end{lemm}
Proof is by explicit computation.

Next we shall restrict ourselves now to the case when $M$ is a $d$-dimensional 
torus, $\T^d$,  and the metric we begin with is a constant, flat metric.

\subsection{The case of $d$-dimensional torus}

Consider a flat $d$-dimensional torus, $\T^d = (S^1)^d$, with a constant  diagonal metric 
$g_{ab} = \delta_{ab}$. We take the usual system of coordinates on the torus (each circle 
parametrized by an angle) and from now on we assume that $\partial_a$ are the associated 
derivations. Take as $\H_0$ the Hilbert space of square summable functions with respect to 
the metric measure. An immediate consequence of Lemma \ref{L1} is:
\begin{lemm}\label{L2}
Let $h$ be a positive function on the torus $\T^n$ . The following operator:
$$ \Delta_h = \sum_{a=1}^n h^{-\frac{d}{2}} \partial_a 
( h^{d-2} \partial_a ) h^{-\frac{d}{2}} , $$
is unitarily equivalent to the Laplace operator of the conformally rescaled
metric $h^2 \delta_{ab}$.
\end{lemm}
This formula has been generalized to the noncommutative case in dimension
$4$ by \cite{FaKh4} to compute the curvature of the conformably rescaled Laplace operator 
on the noncommutative four-torus. However, even though the noncommutative generalization 
of the above prescription for the conformally rescaled Laplace operator makes sense, it does 
not exclude the possibility that the quantity computed is not exactly the scalar curvature.  The 
reason for this is the existence of torsion and the possibility that the above operator might not
be torsion-free in the noncommutative generalization. 

\subsection{Laplace-type operators on noncommutative tori}
 
We shall shall investigate the family of operators, which are noncommutative
generalizations of the above Laplace operator and which differ from them by terms 
of lower order. This guarantees that their principal symbol is unchanged and hence 
using the natural (albeit naive) notion of noncommutative metric we could say both 
operators determine the same metric. We begin with the following definition:

\begin{defi}
Let us take $h \in \A(\T^d_\theta)$ to be a positive element with a bounded inverse and
take the following densely defined operator on $\H$:
\begin{equation}
 \Delta_h = \sum_{a=1}^n h^{-\frac{d}{2}} \delta_a ( h^{d-2} \delta_a ) h^{-\frac{d}{2}} , 
\end{equation}
to be the Laplace operator on $d$-dimensional noncommutative torus.
\end{defi}
In both cases of $d=2$ and $d=4$ this has been studied as the Laplace operator of the
conformally rescaled noncommutative torus. 

The family which we intend to investigate now is,
\begin{defi}\label{D1}
A generalized family of Laplace operator for the conformally rescaled 
metric over the torus has a form:
\begin{equation}
\Delta = \Delta_h +  \sum_{a=1}^n \left( T^a \delta_a + \frac{1}{2} \delta_a(T^a) \right) + X , 
\end{equation}
where $T^a$ and $X$ are some selfadjoint elements of $\A(\T^d_\theta)$. 
This form could be rewritten as:
\begin{equation}
\Delta= h^{-2} \Big( \sum_a \delta_a^2 \Big) + \sum_a  Y^a \delta_a  + \Phi,
\label{mLa}
\end{equation}
where
$$ Y^a = h^{\frac{d}{2}-2} \delta_a (h^{-\frac{d}{2}}) 
+ h^{-\frac{d}{2}} \delta_a (h^{\frac{d}{2}-2}) + T^a,
$$
and
$$ \Phi =  h^{\frac{d}{2}-2} \left( \sum_a \delta_a^2 (h^{-\frac{d}{2}}) \right) + 
h^{-\frac{d}{2}} \left( \sum_a \delta_a(h^{d-2}) \delta_a(h^{-\frac{d}{2}}) \right) + \frac{1}{2} 
\sum_a \delta_a T^a + X.
$$
\end{defi}

The above Laplace-type operator (\ref{mLa}) has the following symbol:
\begin{equation}
\rho(\Delta) = h^{-2} \xi^2 + \sum_a Y^a \xi_a + \Phi, 
\label{mLaSy}
\end{equation}

From now on, we shall work only with the symbol $\rho(\Delta)$ (\ref{mLaSy}).

\section{Wodzicki Residue in dimension $4$}

Our aim will be to compute the Wodzicki residue of $\Delta^{-\frac{d}{2}+k}$ in the case of 
$d=4$ for $k=0,1$. As the only difficulty in considering the general case is purely computational, 
we postpone it for future work, concentrating instead on the relevant case of $d=4$. The 
significance of these computations lies in the classical relation between Wodzicki residue of the 
inverse of the Laplace operator (in the sense of pseudodifferential calculus) and the scalar of 
curvature in the classical case \cite{Gilkey}. We shall see, whether this extends to the 
noncommutative case. For simplicity we keep $X=0$, focusing on the 
parameters $h$ and $T_a$.

\subsection{The symbol of $\Delta^{-2}$ and $\Delta^{-1}$ }

We fix here $d=4$. 

\begin{lemm}
The Wodzicki residue of $\Delta^{-2}$ depends only on $h$:
$$ \hbox{Wres}(\Delta^{-2}) = 2\pi^2 \, \mathfrak{t}(h^{4}). $$
\end{lemm}
\begin{proof}
It is sufficient to observe that the symbol $\rho(\Delta^{-2})$ starts with a homogeneous
symbol of order $-4$, which is exactly $h^4 |\xi|^4$. Hence, computing the Wodzicki 
residue as in proposition (\ref{wprop}) gives the above result. 
\end{proof}

To compute the Wodzicki residue of $\Delta^{-1}$ we need to calculate further
terms of the its symbol.  Observe, that for a differential operator of degree $2$ its 
symbol (split into part of homogenoeus degrees) reads:
$$ a_2 + a_1 + a_0, $$ 
where $a_k$ is homogeneous of degree $k$, and its inverse (paramatrix in the pseudodifferential 
calculus) is,
$$ b = b_0 + b_1 + b_2 + \ldots, $$
where each $b_k$ is homogeneous of degree $-2-k$ and could be iteratively computed from the 
following sequence of identities, which arise from comparing homogenous terms of the 
product $\Delta^{-1}$ and $\Delta$ using (\ref{mLaSy}):
\begin{equation}
\begin{aligned}
&b_0 a_2 = 1, \\
&b_1 a_2 + b_0 a_1 +  \partial_k(b_0) \delta_k(a_2) = 0, \\
&b_2 a_2 + b_1 a_1 + b_0 a_0 +  \partial_k(b_0) \delta_k(a_1) 
+  \partial_k(b_1) \delta_k(a_2) 
+ \oh \partial_k \partial_j (b_0) \delta_k \delta_j (a_2) = 0. 
\end{aligned}
\end{equation} 
The relations could be solved explicitely, giving:
\begin{equation}
b_2= - \left( b_0 a_0 b_0 + b_1 a_1 b_0 + \partial_j(b_0) \delta_j(a_1) b_0 +
 \partial_j(b_1)\delta_j(a_2)b_0  + \oh \partial_{jk}(b_0)\delta_j \delta_j(a_2) \right),
\label{b2}
\end{equation}
where 
\begin{equation}
\begin{aligned}
b_1 &= -\left( b_0 a_1 b_0  + \partial_k(b_0) \delta_k(a_2) b_0 \right), \\
b_0 &= ( a_2 )^{-1},
\end{aligned}
\label{b0b1}
\end{equation}

To simplify the notation above and in the remaining part of the note we use the Einstein notation 
(implicit summation over repeated indices). 

\begin{lemm} \label{L3}
For the pseudodifferential operator (\ref{mLaSy}) we have:

$$
\begin{aligned}
b_2(T,h)(\xi) = |\xi|^{-4} & 
  \left( \frac{1}{4} h^2 T_a h^2 T_a h^2 + \frac{1}{4} h^2 T_a \delta_a(h^2) \right. \\
& \left. \;\;\; - \frac{1}{4} \delta_a(h^2) h^{-2} \delta_a(h^2)  - \frac{1}{4} \delta_a(h^2) T_a  h^2  
          + \frac{1}{2} \delta_{aa}(h^2) \right)
\end{aligned}
$$
which leads to
$$
\begin{aligned}
\hbox{Wres}(\Delta)  = 2 \pi^2  & \left( \mathfrak{t}( h^2 T_a h^2 T_a h^2 )  \right. \\
& \left.  + \frac{1}{4} \mathfrak{t}( h^2 [T_a, \delta_a(h^2)])  \right. \\
& \left.  - \frac{1}{4} \mathfrak{t}( \delta_a(h^2) h^{-2} \delta_a(h^2)) \right).
\end{aligned}
$$
where we have used the trace property of $\mathfrak{t}$ and its invariance with respect to the
action of the derivations. 
\end{lemm}

Before we pass to the interpretation of the above result, let us consider the classical limit $\theta=0$.

\subsection{The commutative case}
We assume here that $\theta=0$, so $h$ and $T_a$ are smooth functions on a torus, which commute with 
each other (and with their derivations). 

\begin{lemm}
For the commutative torus the Wodzicki residue of $\Delta^{-1}$ is:
$$ 
\hbox{Wres}(\Delta^{-1}) = 2\pi^2 \int_{\T^4} \left( h^{6} (T_a T_a) - \delta_a(h)\delta_a(h) \right) dV,
$$
and for a fixed $h$ the term reaches an absolute extremum if and only if $T_a=0$, 
which has the interpretation of torsion-free Laplace operator.
\end{lemm}

Observe that the classical results of Kastle and Kalau-Walze \cite{Ka,KaWa} give (for Laplace-Beltrami 
operator):
$$ \hbox{Wres}(\Delta^{-1}) = 2 \pi^2 \int_M \sqrt{g} \, \left( \frac{1}{6} R \right), $$
where $R$ is the scalar curvature.

In the conformally rescaled metric the volume form and the curvature (in dimension $d=4$) are:
$$ \sqrt{g} = h^4, \;\;\;\; R = 6 h^{-3} \delta_{aa}(h), $$
so we obtain the same result.

\subsection{Nonminimal operators and curvature}

In the classical (commutative) situation the additional first-order term in the Laplace-type operator 
$\Delta$ contributes to the Wodzicki residue of $\Delta^{-1}$ with a term proportional to 
$T_a T_a$. As already noted, for {\em minimal operators} (like Laplace-Beltrami) such term vanishes
and, equivalently, one can say that {\em minimal operators} are such Laplace-type operators, which,
at the fixed metric  minimize the Wodzicki residue.

In the noncommutative case we have two terms, one which is quadratic in $T_a$ and the second one,
which is linear in $T_a$ and involves a commutator with $\delta_a(h^2)$. Therefore one can clearly
state the following corollary,

\begin{coro}
A naive generalization of Laplace-Beltrami operator for the conformally rescaled metric 
to the noncommutative case (in dimension $d=4$) as proposed in (3.1) is not a minimal 
operator in the above sense (does not minimize the Wodzicki residue of $\Delta^{-1}$ 
with $h$ fixed).
\end{coro}

A significant consequence of this fact is that no simple identification of the scalar of curvature is 
possible: indeed, if have no means of identifying {\em minimal} (or unperturbed, or torsion-free) 
Laplace-type operators we cannot possibly recover the geometric invariants associated with the 
metric alone.

\section{Conclusions and open problems}

The computations aimed to show that the Wodzicki residue of the noncommutative generalization
of the Laplace-type operator on the $4$-dimensional noncommutative torus with a conformally recalled
metric is a nontrivial functional on the parameters $h$ and $T_a$. There are several interesting
problems, which arise that are linked to our result.

First of all, computations of the Wodzicki residue are closely linked to heat-kernel coefficients. As 
the principal symbol fails to be scalar, this is not happening for the noncommutative tori. An 
interesting point is then to link the Wodzicki residue in this case to the respective heat-kernel 
coefficients, which then, in turn, appear in the spectral action. 

To study the geometric notions like curvature and geometric constructions like scalar curvature
(at least as a noncommutative analogue of the classical one) it is important to identify the class
of minimal Laplace-type operators. Otherwise, the functional would not only depend on the metric
but also on some additional data. Our result shows that the naive generalization of minimal operator
fails to be minimal in the noncommutative case (at least in the proposed sense). 

A natural question, which arises in this context, is about a proper definition of Dirac and Laplace
operators on noncommutative tori. Even though the flat situation appears to be quite well 
understood, even a slight deviation, like the conformal rescaling, discussed in this note, changes
completely the picture. Unlike classical case we still cannot identify the components of such 
operators, which are of intrinsic geometric origin and distinguish them from the additional degrees
of freedom (like torsion).
                  

\end{document}